\documentclass{birkjour}

\usepackage{amssymb}
\usepackage{amsfonts,tracefnt,amsgen,amsmath,amsthm,amscd,amssymb,stmaryrd, commath, mathrsfs, enumerate}
\usepackage{ mathrsfs,color}
\usepackage{commath}
\usepackage{units}
\usepackage{mathbbol}
\usepackage{esint}

\newtheorem{thm}{Theorem}[section]
\newtheorem{lem}[thm]{Lemma}

\newtheorem{cor}[thm]{Corollary}

\theoremstyle{definition}
\newtheorem{defi}[thm]{Definition}

\newtheorem{remark}[thm]{Remark}
\newcommand{\CC}{{\mathbb C}}

\newcommand{\p}{{p(\cdot)}}
\newcommand{\q}{{q(\cdot)}}

\newcommand{\pprime}{{p'(\cdot)}}

\newcommand{\essi}{\operatornamewithlimits{ess\,inf}}
\newcommand{\esss}{\operatornamewithlimits{ess\,sup}}
\providecommand{\norm}[1]{\lVert#1\rVert}

\begin{document}

\title[Variable Exponent Fock Spaces]{Variable Exponent Fock Spaces}

\author[G. A. Chac\'on]{Gerardo A. Chac\'on\footnote{Supported by research project 2015004: ``Contribuciones a la teor\'ia de operadores en espacios de funciones anal\'iticas''. Universidad Antonio Nari\~no}}

\address{Universidad Antonio Nari\~no, Bogot\'a, Colombia} 
\email{gerardoachg@uan.edu.co}

\author[G. R. Chac\'on]{Gerardo R. Chac\'on}
\address{Gallaudet University, Department of Science,  Technology, and Mathematics,  800 Florida Ave. NE, Washington, DC 20002, USA}
\email{gerardo.chacon@gallaudet.edu}

\begin{abstract}
In this article we introduce Variable exponent Fock spaces and study some of their basic properties such as the boundedness of evaluation functionals, density of polynomials, boundedness of a Bergman-type projection and duality.
\end{abstract}



\maketitle

\section{Introduction}
Variable Exponent Lebesgue spaces  are a generalization of classical Lebesgue spaces $L^p$  in which the exponent $p$ is a measurable function. Such spaces where introduced by  Orlicz \cite{orlicz} and developed by Kov\'a{\v{c}}ik and R\'akosn\'{i}k \cite{kovrak}. For details on such spaces one can refer to \cite{CUF,die_has,kovrak,rafroj2014} and references therein.    

Recently, there has been some interest in variable exponent function spaces consisting of analytic functions defined on a complex domain. For example, in \cite{KP1} and \cite{KP2} variable exponent Hardy spaces of analytic functions in the unit disk are considered. In \cite{KS} a version of $BMO$ spaces with variable exponents is considered. Bergman spaces with variable exponents have been studied in \cite{charaf2014,  charaf2016b, charaf2016}.

In this article, we are interested in studying Fock spaces under the context of variable exponents. Those are spaces of entire functions that belong to a weighted $L^\p$ space with respect to the Gaussian measure. The classical case have received a great interest in the last years as can be seen in \cite{FS}. We will make extensive use of such reference throughout this article. 

One property of variable exponent Fock spaces that makes its study relevant is that the Gaussian measure is not a Muckenhoupt weight, this makes it difficult to use one important resource in the theory of variable exponent Lebesgue spaces: the boundedness of the maximal function. 

The article is distributed as follows. In next section we will present some preliminary concepts and results as well as the notation that will be used throughout  the rest of the article. In Section 3, we will first prove a version of H\"older's inequality to obtain an equivalent norm in variable exponent Fock spaces. Then we will show that evaluation functionals are bounded and prove an inclusion relation between variable exponent Fock spaces with different exponents. Finally we characterize the dual space.

For the rest of the paper, we will use the notation $a\lesssim b$ if there exists a constant $C>0$, independent of $a$ and $b$,  such that $a\leqslant Cb$. Similarly, we use $a\asymp b$ if we have  $a\lesssim b \lesssim a$.

\section{Preliminaries}

\subsection{Fock Spaces}
 We will need  some properties of the classical Fock spaces that we include here for the sake of completeness. This results will be taken from \cite{FS}.

\begin{defi}
Let $A$ denote the Lebesgue area measure defined on the complex plane $\CC$ and let $1\leq p<\infty$. Denote as $L^p=L^p(\CC)$ the Lebesgue space of $p$-integrable functions with respect to the measure $A$. We will  denote as $\mathcal{L}^p=\mathcal{L}^p(\CC)$ the Banach space of all measurable functions $f:\CC\to\CC$ such that 
\[
\|f\|_{{{L}^p}}^p=\frac{p}{\pi}\int_\CC \left|f(z)e^{-\frac{\alpha}{2}|z|^2}\right|^p\,\dif A(z)<\infty.
\] In other words, $\mathcal{L}^p(\CC)$ consists of all functions $f$ such that $fe^{-|\cdot|^2}\in L^p(\CC)$.
The  {\itshape Fock space} $\mathcal{F}^p$ is defined as the space of all entire functions that belong to $\mathcal{L}^p$.  
\end{defi}

\begin{remark}
In the literature, the $\mathcal{L}^p$ spaces just defined are known as {\itshape weighted} $L^p$ spaces in which the weight (in this case the Gaussian weight) acts as a {\itshape multiplier}. 
\end{remark}

For every $z\in\CC$, the \textit{evaluation functional} $\gamma_z:\mathcal{F}^{p}\longrightarrow \CC$ defined as 
\begin{equation}\label{eq:evaluation_constant}
\gamma_z(f):=f(z)
\end{equation}
is bounded since the following inequality holds for every $f\in \mathcal{F}^p$.

\begin{equation}\label{eq:ineq_eval_constant}
|f(z)|\leq \|f\|_{{\mathcal{F}^p}}e^{|z|^2}.
\end{equation}

The space $\mathcal{F}^p$ is a closed subspace of $\mathcal{L}^p$ and consequently it is a Banach space.

In the case $p=2$, the Fock space $\mathcal{F}^2$ is a Hilbert space with inner product 
\begin{equation}\label{eq:inner_prod}
\langle f,g\rangle =\frac{2}{\pi}\int_\CC f(z)\overline{g(z)}e^{-2|z|^2}\dif A(z).
\end{equation}
As a consequence of Riesz representation theorem, for every $z\in \CC$, there exists an element $K_z\in \mathcal{F}^2$ such that for every $f\in \mathcal{F}^2$, 
\begin{equation}\label{eq:reproducing}
\langle f,K_z\rangle=f(z).
\end{equation}
The functions $K_z$ are called {\itshape reproducing kernels} and are given by
\begin{equation}\label{eq:rep}
K_z(w)=e^{2w\overline{z}}.
\end{equation}

It is shown in \cite{tung} that the reproducing formula \eqref{eq:reproducing} holds for general functions $f\in \mathcal{F}^p$ in the sense that 
\begin{equation}\label{eq:represent}
f(z)=\frac{2}{\pi}\int_\CC f(w)e^{2\overline{w}z}e^{-2|w|^2}\dif A(w)
\end{equation}

\subsection{Generalized Orlicz Spaces}
We will also need  some results from the theory of generalized Orlicz spaces that will be presented as in \cite{HHK}.

\begin{defi}
A function $\varphi:\CC\times [0,\infty)\to [0,\infty)$ is said to belong to the class $\Phi$ if for every $t\in [0,\infty)$ the function $\varphi(\cdot,t)$ is measurable and for every $z\in\CC$, the function $\varphi(z,\cdot)$ satisfies the  following conditions:
\begin{enumerate}[(i)]
\item  $\varphi(z,\cdot)$ is increasing 
\item  $\varphi(z,\cdot)$ is left-continuous 
\item  $\varphi(z,\cdot)$ is convex
\item  $\varphi(z,0)=\lim_{t\to 0^+}  \varphi(z,t)=0$ and  $\lim_{t\to \infty}  \varphi(z,t)=\infty$
\end{enumerate}
\end{defi}

Some examples of functions of the class $\Phi$ are:
\[\varphi_1(z,t)=t^{p(z)}\qquad \varphi_2(z,t)=\frac{1}{p(z)}t^{p(z)}\qquad \varphi_3(z,t)=t^{p(z)}e^{-|z|^2p(z)}\] where $p:\CC\to[1,\infty)$ is a measurable function.

\begin{defi}\label{def:general_orlicz}
Given a function $\varphi$ in the class $\Phi$, define for every measurable function $f:\CC\to\CC$, the modular
 \begin{equation}\label{eq:modular_phi}
 \rho_{\varphi}(f)=\int_\CC \varphi\left(z,|f(z)|\right)\dif A(z)
 \end{equation}
The {\itshape Generalized Orlicz space} $L^\varphi$ is defined as the set of all measurable functions $f$ such that $\rho_\varphi(\lambda f)<\infty$ for some $\lambda>0$. $L^\varphi$ is a Banach space when eqquiped with the {\itshape Luxemburg-Nakano norm:} 
\begin{equation}\label{eq:norm_phi}
\|f\|_{L^\varphi}= \inf \left\{ \lambda > 0:
\rho_{\varphi}\left(\frac{f}{\lambda}\right)\leqslant 1 \right\}
\end{equation}
\end{defi}

In the case in which $\varphi=\varphi_1$, we call $L^\varphi=L^{\p}$ a {\itshape variable exponent Lebesgue space}. The basics on the subject may be found in the monographs \cite{CUF,die_has}. 
Denote as
$p^+=    \esss_{z\in\CC}p(z)$ and $ p^-=\essi_{z\in\CC}p(z)$.
The measurable function $p:\CC \rightarrow [1,\infty)$, is called a \textit{variable exponent}, and  the set of all variable exponents with $p^+<\infty$ is denoted as $\mathscr P(\CC)$. 

We will be particularly interested in the case in which $\varphi=\varphi_3$. We will denote such generalized Orlicz space as $\mathcal{L}^\p$ to stress the dependence of the exponent and we use the calligraphic $\mathcal{L}$ in order to differentiate from the unweighted Lebesgue space.

Some regularity conditions will also be needed to the exponent $p$.

\begin{defi}
	A function $p:\CC \longrightarrow [1,\infty)$ is said to be \textit{$\log$-H\"older continuous}  or to satisfy the \textit{Dini-Lipschitz condition} on $\CC$ if there exists a positive constant $C_{\log}$ such that
\begin{equation}
	|p(z)-p(w)|\leqslant \frac{C_{\log}}{\log\left( 1/|z-w| \right)},
	\label{eq:local}
\end{equation}
for all $z, w \in \CC$ such that $|z-w|<1/2$. 
The function $p$ is said to satisfy the \textit{$\log$-H\"older decay condition}, if there exists $p_\infty \in \mathbb [1,\infty)$ and a positive constant $C$ such that
\begin{equation}
	|p(z)-p_\infty|\leqslant \frac{C}{\log\left( e+|x| \right)}
	\label{eq:decay}
\end{equation}
for all $z \in \CC$. We say that the function $p$ is \textit{globally $\log$-H\"older continuous in $\CC$} if it satisfies \eqref{eq:local} and \eqref{eq:decay}. We denote by $\mathscr P^{\log} (\CC)$ the set of all globally $\log$-H\"older continuous functions in $\CC$ for which 
\[
1< p_- \leqslant p_+ < \infty.	
\]
\end{defi}
It is known that condition \eqref{eq:local} is equivalent to the following intequality 
\[
|B|^{p_-(B)-p_+(B)}\le C,
\] to hold for all balls $B\subset\CC$. Here $|B|$ stands for the Lebesgue measure of $B$, $p^+(B)=    \esss_{z\in B}p(z)$ and $ p^-(B)=\essi_{z\in B}p(z)$.

\subsection{Variable Exponent Fock Spaces}
We are ready to introduce the main concept of this article.

\begin{defi}
Let $p:\CC\to[1,\infty)$ be a measurable function in $\mathscr{P}(\CC)$. The {\itshape Variable Exponent Fock Space} $\mathcal{F}^\p$ is defined as the set of all entire functions that belong to the generalized Orlicz space $\mathcal{L}^\p$.

In othe words, $\mathcal{F}^\p$ is the set of entire functions such that \[\int_\CC |f(z)|^{p(z)} e^{-|z|^2p(z)}\dif A(z)<\infty.\]
In order to have a better correspondence with the definition of Fock spaces for constant exponents, a modification will be made to the modular defined in \eqref{eq:modular_phi}. 
We will denote as \[C_\p:=\int_\CC e^{-|z|^2p(z)}\dif A(z)\] and define the modular \[\rho_\p(f)=C_\p^{-1}\int_\CC |f(z)|^{p(z)} e^{-|z|^2p(z)}\dif A(z).\] The choice of the constant is made so that $\rho_\p(1)=1$.
\end{defi}

\begin{remark}
An entire function $f$ belongs to $\mathcal{F}^\p$ if, and only if the function $z\mapsto f(z)e^{-|z|^2}$ belongs to $L^\p$.
\end{remark}

%We finish this section with some basic properties of the modular that will be used during the article.
%
%\begin{pro}\label{prop:properties_modular}
%Suppose that $p\in \mathscr{P}(\CC)$. Then the following properties hold:
%\begin{enumerate}
%\item For every measurable function $f$, $\rho_\p(f)=\rho_\p(|f|)$ and $\rho_\p(f)\geq 0$.
%\item $\rho_\p(f)=0$ if, and only if $f=0$ almost everywhere in $\CC$.
%\item If $\rho_\p(f)<\infty$ then $f(z)<\infty$ almost everywhere in $\CC$.
%\item $\rho_\p$ is convex: given $\alpha,\beta\geq 0$ with $\alpha+\beta=1$, it holds that
%$$
%\rho_\p(\alpha f+\beta g)\leq \alpha\rho_\p(f)+\beta \rho_\p(g).
%$$
%\item $\rho_\p$ is order-preserving: $\rho_\p(f)\leq \rho_p(g)$ if $|f(z)|\leq |g(z)|$ almost everywhere.
%\item If $\rho_\p(f/\lambda_0)<\infty$ for some $\lambda_0>0$, then the function $\lambda\mapsto \vrho(f/\lambda)$ is continuous and decreasing in $[\lambda_0,\infty)$ and $\rho_\p(f/\lambda)\to 0$ as $\lambda\to \infty$.
%\item \label{prop_less1} $\rho_\p(\alpha f)\leq \alpha \rho_p(f)$ if $0<\alpha<1$.
%\item $\alpha \rho_\p(f)\leq \rho_\p(\alpha f)$ if $\alpha>1$.
%\item If $\|f\|_{\mathcal{L}^\p}=1$, then $\rho_\p(\|f\|_{\mathcal{L}^\p}^{-1}f)=1$.
%\item\label{prop_modular_norm} Suppose $f\in \mathcal{L}^p$. If $\|f\|_\p\leq 1$, then $\rho_\p(f)\leq\|f\|_{\mathcal{L}^\p}$; if $\|f\|_{\mathcal{L}^\p}>1$, then $\|f\|_{\mathcal{L}^\p}\leq\rho_\p(f)$.
%\end{enumerate}
%\end{pro}

%Notice that as a consequence of property \ref{prop_modular_norm} above, we have that \[\|1\|_{\mathcal{L}^\p}=1\]

\subsection{Weighted Fock Spaces}
 \begin{defi}
 Fix $r>0$ and $1<p<\infty$. Let $A_{p,r}$ denote the class of weights $w:\CC\to [0,\infty)$ such that \[\sup_{z\in\CC} \left(\frac{1}{|B(z,r)|}\int_{B(z,r)}w\dif A\right)\left(\frac{1}{|B(z,r)|}\int_{B(z,r)}w^{-\frac{1}{p-1}}\dif A\right)^{p-1}\leq C_r\] for some $0<C_r<\infty$. 

We will say that $w$ belongs to the Muckenhoupt class $A_1$ if \[\esss_{z\in\CC}\frac{Mw(z)}{w(z)}<\infty\] where $M$ denotes the Hardy-Littlewood maximal operator.

It is shown in \cite[section 4.2]{CUF}  that  $A_1\subset A_{p,r}$ for every $1<p<\infty$ and $r>0$. 
 \end{defi}

 Given a weight $w$, we will denote as $L^p(w)$  the weighted Lebesgue space whereas $\mathcal{L}^p(w)$ will denote the weighted $\mathcal{L}^p$ space. The following theorem is shown in \cite{Isr2014}.

\begin{thm}\label{thm:ISRA}
The following are equivalent for any weight $w$ on $\CC$:
\begin{enumerate}[(i)]
\item $w\in A_{p,r}$ for some $r>0$.
\item The operator $H:L^p(w)\to L^p(w)$ defined as \[Hf(z)=\int_\CC e^{-|z-u|^2}f(u)\dif A(u)\] is bounded.
\item The operator $P:\mathcal{L}^p(w)\to\mathcal{F}^p(w)$  defined as \[Pg(z)=\frac{2}{\pi}\int_\CC g(w)e^{2\overline{w}z}e^{-2|w|^2}\dif A(w)\] is bounded.
\item $w\in A_{p,r}$ for all $r>0$.
\end{enumerate}
\end{thm}

\begin{remark}\label{cor:J_op}
The operator $P$ is bounded on $\mathcal{L}^p(w)$ if and only if the operator \[\fullfunction{J}{\mathcal{L}^p(w)}{\mathcal{L}^p(w)}{g}{\int_\CC \left|g(z)e^{2\overline{z}w}e^{-2|z|^2}\right|\dif A(z)}\] is bounded.
\end{remark}
%
%\begin{proof}
%Clearly, if $J$ is bounded, then $P$ is bounded on $\mathcal{L}^p(w)$. On the other hand, $P$ being bounded on $\mathcal{L}^p(w)$ is equivalent, by the previous theorem, to $H$ being bounded on $L^p(w)$ so if $f\in L^p(w)$, then 
%\[
%\int_\CC\left|\int_\CC e^{-|z-u|^2}f(u)\dif A(u)\right|^p w(z)\dif A(z)\lesssim \|f\|_{L^p(w)}
%\] and consequently
%\[
%\int_\CC\left|\int_\CC f(u) e^{-|z|^2}|e^{2z\overline{u}}|e^{-|u|^2}\dif A(u)\right|^p w(z)\dif A(z)\lesssim \|f\|_{L^p(w)}.
%\]
%Now writing $f(u)=g(u)e^{-|u|^2}$, with $g\in \mathcal{L}^p(w)$, the previous equation becomes
%\[\int_\CC |Jg(z)|^p e^{-|z|^2p} w(z)\dif A(z)\lesssim \|g\|_{\mathcal{L}^p(w)}\]
%
%\end{proof}

The following proposition is a version of Rubio de Francia extrapolation result in the framework of variable Lebesgue spaces. 

\begin{thm}[Thm. 5.24 in \cite{CUF}]\label{thm:extrap}
Given a set $\Omega$, suppose that for some $p_0 \geq 1$ the family $\mathcal D$ is such that for all $w\in A_1$,
\begin{equation}
	\int_{\Omega} F(x)^{p_0} w(x) \dif x \leq C_0 \int_{\Omega} G(x)^{p_0}w(x) \dif x, \quad (F,G) \in \mathcal D.
	\label{eq:extrapolation}
\end{equation}
Given $p \in \mathscr{P}(\Omega)$, if $p_0\leq p_-\leq p_+ < \infty$ and the Hardy-Littlewood maximal operator is bounded on $L^{(p(\cdot)/p_0)^\prime}(\Omega)$, then
\begin{equation}
	\norm{F}_{L^{p(\cdot)}(\Omega)} \leq C_{p(\cdot)} 	 \norm{G}_{L^{p(\cdot)}(\Omega)}
	\label{eq:extrapolationnorm}
\end{equation} for every $(F,G) \in \mathcal D$.
\end{thm}

\section{Some Properties of Variable Exponent Fock Spaces}
In this section we will start by showing a version of  H\"older inequality for $\mathcal{L}^\p$ spaces that will be useful in this specific context. It is important to notice that, as  generalized Orlicz spaces, a H\"older inequality and a duality result already exists, however here we prove a different version that better suits our purposes.

\begin{thm}[H\"older's Inequality]\label{thm:Holder}
Suppose $p\in \mathscr{P}(\CC)$ and let $p':\CC\to (1,\infty)$ be such that $\frac{1}{p(z)}+\frac{1}{p'(z)}=1$ for all $z\in \CC$. Then if $f\in  \mathcal{L}^{p(\cdot)}$ and $g\in  \mathcal{L}^{p'(\cdot)}$, we have that  \[\left|\int_\CC f(z)\overline{g(z)}e^{-2|z|^2}\dif A(z)\right|\leq 2\|f\|_{{\mathcal{L}^\p}}\|g\|_{{\mathcal{L}^\pprime}}\]
\end{thm}

\begin{proof}
It is clear that if any of the norms in the right-hand side is equal to zero, then the inequality holds. So suppose both norms are nonzero, and use Young's inequality to obtain \[\frac{|f(z)|}{\|f\|_{{\mathcal{L}^\p}}}\frac{|g(z)|}{\|g\|_{{\mathcal{L}^\pprime}}}e^{-2|z|^2}\leq \frac{|f(z)|^{p(z)}e^{-p(z)|z|^2}}{p(z)\|f\|_{{\mathcal{L}^\p}}}+\frac{|g(z)|^{{\mathcal{L}^\pprime}} e^{-p'(z)|z|^2}}{p'(z)\|g\|_{p'(\cdot)}}\] and the inequality follows. 
\end{proof}

With the previous inequality in hand, we can define an alternative norm on $\mathcal{L}^{p(\cdot)}$ as \[\interleave f\interleave_{{\mathcal{L}^\p}}=\sup\left\{\frac{2}{\pi}\left|\int_\mathbb{C}f(z)\overline{g(z)}e^{-2|z|^2}\dif A (z)\right|: g\in \mathcal{L}^{p'(\cdot)},~~ \rho_{p'(\cdot)}(g)\leq 1\right\}.\]

We gather some properties of this norm in the following lemma whose proof is routine. We aim to prove that norm $\interleave\cdot\interleave_{{\mathcal{L}^\p}}$ is equivalent to norm $\|\cdot\|_{{\mathcal{L}^\p}}$.

\begin{lem}
Suppose that $p\in \mathcal{P}(\CC)$ and let $p':\CC\to (1,\infty)$ be such that $\frac{1}{p(z)}+\frac{1}{p'(z)}=1$ for all $z\in \CC$. Then 
\begin{align*}
\interleave f\interleave_{{\mathcal{L}^\p}} &=\sup\left\{\frac{2}{\pi}\int_\mathbb{C}|f(z)||g(z)|e^{-2|z|^2}\dif A (z): g\in \mathcal{L}^{p'(\cdot)},~~ \rho_{p'(\cdot)}(g)\leq 1\right\}\\
&= \sup\left\{\frac{2}{\pi}\int_\mathbb{C}|f(z)||g(z)|e^{-2|z|^2}\dif A (z): g\in \mathcal{L}^{p'(\cdot)},~~ \rho_{p'(\cdot)}(g)<\infty \right\}
\end{align*}
\end{lem}

\begin{thm}
Suppose that $p\in \mathcal{P}(\CC)$ and let $p':\CC\to (1,\infty)$ be such that $\frac{1}{p(z)}+\frac{1}{p'(z)}=1$ for all $z\in \CC$. Then there exist constants $c>0$ and $C>0$ such that \[c\|f\|_{{\mathcal{L}^\p}}\leq \interleave f\interleave_{{\mathcal{L}^\p}}\leq C\|f\|_{{\mathcal{L}^\p}} \]
\end{thm}

\begin{proof}
Suppose $f\in \mathcal{L}^\p$. Then by H\"older's inequality we have 
\begin{align*}
\interleave f\interleave_{{\mathcal{L}^\p}} &\leq \sup\left\{\frac{2}{\pi}\int_\mathbb{C}|f(z)||h(z)|e^{-2|z|^2}\dif A (z): h\in \mathcal{L}^{p'(\cdot)},~~ \rho_{p'(\cdot)}(h)\leq 1\right\}\\
&\leq \frac{4}{\pi}\sup\left\{\|f\|_{{\mathcal{L}^\p}}\|h\|_{{\mathcal{L}^\pprime}}: h\in \mathcal{L}^{p'(\cdot)},~~ \rho_{p'(\cdot)}(h)\leq 1\right\}\\
&\leq \frac{4}{\pi}\|f\|_\p.
\end{align*}

On the other hand, suppose that $\interleave f\interleave_{\mathcal{L}^\p}= 1$. Define the function $g$ as \[g(z)=\frac{\pi}{2C_\p}|f(z)|^{p(z)-1}e^{-|z|^2(p(z)-2)}.\] Then 
\begin{align*}
\rho_\pprime(g) &= \frac{\pi}{2C_\pprime C_\p}\int_\CC |f(z)|^{(p(z)-1)p'(z)}e^{-|z|^2(p(z)-2)p'(z)}e^{-|z|^2p'(z)}\dif A(z)\\
                         &=\frac{\pi}{2C_\pprime C_\p}\int_\CC |f(z)|^{p(z)}e^{-|z|^2p(z)}\dif A(z)\\
                         &\leq \frac{\pi}{2C_\pprime}\rho_\p(f).
\end{align*}
Thus, $g\in \mathcal{L}^\pprime$ and consequently 
\begin{align*}
\rho_\p(f) &=C_\p^{-1}\int_\CC |f(z)||f(z)|^{p(z)-1}e^{-|z|^2(p(z)-2)}e^{-2|z|^2}\dif A(z)\\
 	       &=\frac{2}{\pi}\int_\CC |f(z)||g(z)|e^{-2|z|^2}\dif A(z)\\
                &\leq \interleave f\interleave_{\mathcal{L}^\p} =1.
\end{align*}

Thus \[\rho_\p\left(\frac{f}{\interleave f\interleave_{\mathcal{L}^\p}}\right)=\rho_\p(f)\leq 1\] which implies that \[\|f\|_{\mathcal{L}^\p}\leq \interleave f\interleave_{\mathcal{L}^\p}.\] The general case $\interleave f\interleave_{\mathcal{L}^\p}\neq 1$ follows from the homogeneity of the norm.
\end{proof}

We pass now to prove that evaluation functionals are continuous. Given $z\in\CC$ the evaluation functional is defined as \[\fullfunction{\gamma_z}{\mathcal{F}^\p}{\CC}{f}{f(z)}.\] We will need a previous lemma.

\begin{lem}
Let $f:\CC\to\CC$ be an entire function and let $R>0$. Then for every $z\in \CC$, \[|f(z)|e^{-|z|^2}\leq \frac{e^{R^2}}{R^2\pi}\int_{B(z,R)}|f(w)|e^{-|w|^2}\dif A(w)\]
\end{lem}

\begin{proof}
Fix $z\in\CC$ and define the entire function \[\fullfunction{g}{\CC}{\CC}{w}{f(w+z)e^{-2\overline{z}w}}.\] By the mean value theorem we have that 
%\[g(0)=\frac{1}{2\pi}\int_0^{2\pi}g(re^{i\theta})\dif\theta\] and consequently \[|g(0)|\leq \frac{1}{2\pi}\int_0^{2\pi}|g(re^{i\theta})|\dif\theta.\] Multiplying both sides by $r>0$ and integrating from $0$ to $R$, we have 
\[|g(0)|\leq \frac{1}{R^2\pi}\int_{B(0,R)}|g(w)|\dif A(w).\] Writing the equation in terms of $f$ and multiplying by $e^{-|z|^2}$ we get 
\begin{align*}
|f(z)|e^{-|z|^2}&\leq \frac{1}{R^2\pi}\int_{B(0,R)}|f(w+z)||e^{-2\overline{z}w}|e^{-|z|^2}\dif A(w)\\
                      &\leq \frac{e^{R^2}}{R^2\pi}\int_{B(z,R)}|f(w)|e^{-|w|^2}\dif A(w).\\
\end{align*}
\end{proof}

\begin{thm}
Suppose that $p\in \mathscr{P}(\CC)$. Then there exists a constant $C>0$ such that for every $f\in\mathcal{F}^\p$ \[|f(z)|\leq Ce^{|z|^2}\|f\|_{\mathcal{F}^\p}\]
\end{thm}

\begin{proof}
Take $R=1$ in the previous lemma to get \[|f(z)|e^{-|z|^2}\leq \frac{e}{\pi}\int_{B(z,1)}|f(w)|e^{-|w|^2}\dif A(w).\] Now, applying theorem \ref{thm:Holder} we get that if $\frac{1}{p(w)}+\frac{1}{p'(w)}=1$ for all $w\in\CC$, then
\begin{align*}
\int_{B(z,1)}|f(w)|e^{-|w|^2}\dif A(w) &= \int_{\CC}|f(w)|e^{|w|^2}\chi_{B(z,1)}(w)e^{-2|w|^2}\dif A(w)\\
&\leq 2\|f\|_{\mathcal{F}^\p}\|e^{|\cdot|^2}\chi_{B(z,1)}\|_{\mathcal{L}^\pprime}.
\end{align*}
 
% But notice that \[\rho_\pprime\left(e^{|\cdot|^2}\chi_{B(z,1)}\right)=\frac{\pi}{C_\pprime}.\] So if we let $M=\max\{1,\pi C_\pprime^{-1}\}$, then $M^{-1}\rho_{\pprime}\left(e^{|\cdot|^2}\chi_{B(z,1)}\right)\leq 1$ then by property \eqref{prop_less1} in \ref{prop:properties_modular} we have that \[\|e^{|\cdot|^2}\chi_{B(z,1)}\|_{\mathcal{L}^\pprime}\leq M\] 

\noindent and the result follows.
\end{proof}

\begin{cor}
Suppose that $p\in \mathscr{P}(\CC)$. Then $\mathcal{F}^\p$ is a closed subspace of $\mathcal{L}^\p$ and hence it is a Banach space.
\end{cor}

%\begin{proof}
%As a consequence of the previous theorem, convergence in  $\mathcal{F}^\p$ implies uniform convergence on compact subsets of $\CC$ and consequently $\mathcal{F}^\p$-limits of sequences in $\mathcal{F}^\p$ are analytic. 
%\end{proof}

As another consequence of the previous theorem, we can now show a relation between $\mathcal{F}^\p$ spaces.

\begin{thm}\label{thm:subset}
Suppose $p,q\in\mathscr{P}(\CC)$ and $p(z)\leq q(z)$ for every $z\in\CC$. Then $\mathcal{F}^\p\subset\mathcal{F}^\q$ and moreover, for every $f\in \mathcal{F}^\p$, \[\|f\|_{\mathcal{F}^\q}\lesssim \|f\|_{\mathcal{F}^\p}.\]
\end{thm}

\begin{proof}
Suppose that $f\in \mathcal{F}^\p$ and $\|f\|_{\mathcal{F}^\p}\leq 1$. By the previous theorem, we can choose a constant $C_1>1$ such that for every $z\in \CC$, $|f(z)|\leq C_1e^{|z|^2}\|f\|_{\mathcal{F}^\p}$. Thus,
\begin{align*}
\rho_\q(f) &= C_\q^{-1}\int_\CC |f(z)|^{q(z)}e^{-|z|^2q(z)}\dif A(z)\\
	       &= C_\q^{-1}\int_\CC |f(z)|^{p(z)}|f(z)|^{q(z)-p(z)}e^{-|z|^2q(z)}\dif A(z)\\
	       &\leq C_\q^{-1} \int_\CC |f(z)|^{p(z)}\left|C_1e^{|z|^2}\right|^{q(z)-p(z)}e^{-|z|^2q(z)}\dif A(z)\\
	       %&\leq \frac{C_1^{q^+-p^-}}{C_\q} \int_\CC |f(z)|^{p(z)}e^{-|z|^2p(z)}\dif A(z)\\
	       &=\frac{C_1^{q^+-p^-}C_\p}{C_\q}\rho_\p(f).\\
	       %&\leq M\|f\|_{\mathcal{F}^\p}\leq M
\end{align*}
 
 Now for a general $f\in\mathcal{F}^\p$, we apply the previous inequality to $f\|f\|_{\mathcal{F}^\p}^{-1}$ to obtain
 \begin{align*}
 \rho_\q\left(\frac{f}{M\|f\|_{\mathcal{F}^\p}}\right) &\leq \frac{1}{M}\rho_\q\left(\frac{f}{\|f\|_{\mathcal{F}^\p}}\right)\\
 	&\leq 1
 \end{align*}
  and consequently, $\|f\|_{\mathcal{F}^\q}\leq M\|f\|_{\mathcal{F}^\p}$.
\end{proof}

Now we define (formally) the projection $P$ of a function $g$ as 
\begin{equation}\label{eq:projection}
 Pg(z)=\frac{2}{\pi}\int_\CC g(w)e^{2\overline{w}z}e^{-2|w|^2}\dif A(w).
\end{equation}
 It is shown in \cite[section 2.2]{FS} that $P$ is a linear operator that maps $\mathcal{L}^p$ onto $\mathcal{F}^p$. We will show the analogous result for the case of variable exponents by
putting Theorems \ref{thm:ISRA} and \ref{thm:extrap} together.

\begin{thm}
Suppose that $p$ belongs to $\mathscr{P}^{\log} (\CC)$. Then the operator $P$ defined in \eqref{eq:projection} is bounded from $\mathcal{L}^\p$ onto $\mathcal{F}^\p$. 
\end{thm}

\begin{proof}
First notice that since $\mathcal{F}^\p\subset\mathcal{F}^{p^+}$, then the representation \eqref{eq:represent} holds for every function $g\in \mathcal{F}^\p$ and consequently $Pg=g$. This implies that $P$ is surjective. 

Now we will use Theorem \ref{thm:extrap}.  Since $p\in\mathscr{P}^{\log} (\CC)$, then the Hardy-Littlewood maximal operator is bounded on $L^{(p(\cdot)/p_0)^\prime}(\CC)$ for any $1<p_0<p^-$ (see \cite[section 5.4]{CUF}). Moreover, it was mentioned before that if $w$ is any weight in $A_1$, then it belongs to $A_{p_0,r}$ for any $r>0$.

In consequence, by Theorem \ref{thm:ISRA} we have that there exists a constant $C>0$ such that for any $g\in \mathcal{L}^{p_0}(w)$ the following inequality holds: \[\int_\CC \left|Pg(z)\right|^{p_0}e^{-|z|^2p_0}w(z)\dif A(z)\leq C\int_\CC \left|g(z)\right|^{p_0}e^{-|z|^2p_0}w(z)\dif A(z).\]
Hence, if we define the family \[\mathscr{D}=\left\{\left(e^{-|\cdot|^2}Pg,e^{-|\cdot|^2} g\right):g\in\mathcal{L}^{p_0}(w)\right\}\] then we have the hypotheses of Theorem \ref{thm:extrap}. Thus, there exists a constant $C_\p>0$ such that \[\|e^{-|\cdot|^2}Pg\|_{L^\p(\CC)}\leq C_\p\|e^{-|\cdot|^2}g\|_{L^\p(\CC)}\] for every $g\in \mathcal{L}^{p_0}(w)$ and therefore \[\|Pg\|_{\mathcal{L}^\p(\CC)}\leq C_\p\|g\|_{\mathcal{L}^\p(\CC)}.\] Finally, by the density of $\mathcal{L}^{p_0}$ in $\mathcal{L}^\p$ we get that the inequality holds for every function in $\mathcal{L}^\p$.
\end{proof}

We are almost ready to show a duality result for variable exponent Fock spaces. This will come as a consequence of the corresponding duality result for $\mathcal{L}^\p$ spaces.

\begin{thm}
Suppose that $p\in\mathscr{P}(\CC)$ and let $p'$ be such that $\frac{1}{p(z)}+\frac{1}{p'(z)}=1$ for all $z\in\CC$. Then the dual space of $\mathcal{L}^\p$ is isomorphic to $\mathcal{L}^{\pprime}$ and every functional $\Lambda\in\left(\mathcal{L}^\p\right)^\ast$  is of the type \[f\mapsto \frac{2}{\pi}\int_\CC f(z)\overline{h(z)}e^{-2|z|^2}\dif A(z)\] and \[\|\Lambda\|\sim\|h\|_{\mathcal{L}^\pprime}.\]
\end{thm}

\begin{proof}
First notice that the operator \[\fullfunction{Q_\p}{L^\p}{\mathcal{L}^\p}{f}{C_\p e^{|\cdot|^2}f}\] satisfies that for every $f\in L^\p$, it holds that \[\rho_\p\left(\frac{Q_\p f}{\lambda}\right)=\rho_{\varphi_3}\left(\frac{f}{\lambda}\right).\] This implies that \[\|Q_\p f\|_{\mathcal{L}^\p}=\|f\|_{L^\p}.\]

Now suppose that $h\in \mathcal{L}^\pprime$ and define the linear functional \[\fullfunction{\Lambda_h}{\mathcal{L}^\p}{\CC}{f}{\frac{2}{\pi}\int_\CC f(z)\overline{h(z)}e^{-2|z|^2}\dif A(z)}.\] By Theorem \ref{thm:Holder} we have that \[|\Lambda_h(f)|\leq \frac{4}{\pi}\|f\|_{\mathcal{L}^\p}\|h\|_{\mathcal{L}^\pprime}\] and consequently $\lambda_h\in \left(\mathcal{L}^\p\right)^\ast$.
On the other hand, suppose that $\Lambda\in \left(\mathcal{L}^\p\right)^\ast$ and define \[\fullfunction{\Gamma}{L^\p}{\CC}{g}{C_\p^{-1}\Lambda Q_\p g}.\] Since $\Lambda$ and $Q_\p$ are bounded, then $\Gamma\in(L^\p)^\ast$ and by the duality result for variable exponent Lebesgue spaces (see for example \cite[Section 2.8]{CUF}), there exists a function $u\in L^\pprime$ such that $\|u\|_{L^\pprime} \sim \|\Gamma\| $ and for every  $g\in L^\p$ \[\Gamma g=\int_\CC g(z)\overline{u(z)}\dif A(z).\] Take $h=\frac{\pi}{2C_\pprime} Q_\pprime u\in \mathcal{L}^\pprime$ and notice that if $f\in\mathcal{L}^\p$, then
\begin{align*}
\Lambda f &= C_\p\Gamma Q_\p^{-1}f\\
               %  &= C_\p\int_\CC \left(Q_\p^{-1}f(z)\right)\overline{u(z)}\dif A(z)\\
                % &= \frac{2 C_\p C_\pprime}{\pi}\int_\CC \left(Q_\p^{-1}f(z)\right)\overline{Q^{-1}_\pprime h(z)}\dif A(z)\\
                 &= \frac{2}{\pi}\int_\CC f(z)\overline{g(z)}e^{-2|z|^2}\dif A(z).
\end{align*}
Moreover, $\|\Lambda\|=\|\Gamma\|\sim\|u\|_{L^\pprime}=\|h\|_{\mathcal{L}^\p}$
\end{proof}

\begin{remark}
In the previous theorem, the constant $\frac{2}{\pi}$ is being considered in order to keep the correspondence with the representation given by equation \eqref{eq:represent}.
\end{remark}

\begin{thm}
Suppose that $p\in\mathscr{P}^{\log}(\CC)$ and let $p'$ be such that $\frac{1}{p(z)}+\frac{1}{p'(z)}=1$ for all $z\in\CC$. Then the dual space of $\mathcal{F}^\p$ is isomorphic to $\mathcal{F}^{\pprime}$ and every functional $\Lambda\in\left(\mathcal{F}^\p\right)^\ast$  is of the type \[f\mapsto \langle f, h\rangle=\frac{2}{\pi}\int_\CC f(z)\overline{h(z)}e^{-2|z|^2}\dif A(z)\] and \[\|\Lambda\|\sim\|h\|_{\mathcal{F}^\pprime}.\]
\end{thm}

\begin{proof}
First, by a similar reasoning as in the previous theorem, for every function $h\in \mathcal{F}^\pprime$ the linear functional \[\fullfunction{\Lambda_h}{\mathcal{F}^\p}{\CC}{f}{\frac{2}{\pi}\int_\CC f(z)\overline{h(z)}e^{-2|z|^2}\dif A(z)}\] is bounded.

On the other hand, if $\Lambda\in\left(\mathcal{F}^\p\right)^\ast$ then since $\mathcal{F}^\p$ is a closed subset of $\mathcal{L}^\p$, then we use Hahn-Banach theorem to extend $\Lambda$ to a bounded linear functional $\tilde{\Lambda}\in \left(\mathcal{L}^\p\right)^\ast$ with $\|\tilde{\Lambda}\|=\|\Lambda\|$. By the previous theorem, there exists a function $\tilde{h}\in \mathcal{L}^\pprime$ such that $\|\tilde{h}\|=\|\tilde{\Lambda}\|$ and for every $f\in \mathcal{L}^\p$, \[\Lambda(f)=\frac{2}{\pi}\int_\CC f(z)\overline{\tilde{h}(z)}e^{-2|z|^2}\dif A(z).\] Take $h=P\tilde{h}$ where $P$ denotes the projection defined in equation \eqref{eq:projection}. Since $P$ is bounded, then $\|h\|_{\mathcal{F}^\p}\lesssim\|\tilde{h}\|_{\mathcal{L}^\p}$. 
Moreover,  if $f\in \mathcal{F}^\p$, then we know that $Pf=f$ and consequently, 
\begin{align*}
\Lambda f &= \tilde{\Lambda}f\\
                 &= \frac{2}{\pi}\int_\CC f(z)\overline{\tilde{h}(z)}e^{-2|z|^2}\dif A(z)\\
                 %&= \frac{4}{\pi^2}\int_\CC \int_\CC f(w)e^{2z\overline{w}}e^{-2|w|^2}\overline{\tilde{h}(z)}e^{-2|z|^2}\dif A(w)\dif A(z)\\
                 &= \frac{4}{\pi^2}\int_\CC f(w) e^{-2|w|^2}\int_\CC e^{2z\overline{w}}\overline{\tilde{h}(z)}e^{-2|z|^2}\dif A(z)\dif A(w)\\
                 %&= \frac{2}{\pi}\int_\CC f(w) e^{-2|w|^2}\overline{P{\tilde{h}(w)}}\dif A(w)\\
                 &= \frac{2}{\pi}\int_\CC f(w) \overline{h(w)}e^{-2|w|^2}\dif A(w).
\end{align*}
Where the use of Fubini's theorem is justified since   \[\int_\CC \left| f(w) e^{-2|w|^2}\int_\CC e^{2z\overline{w}}\overline{\tilde{h}(z)}e^{-2|z|^2}\dif A(z)\right|\dif A(w)\lesssim  \|f\|_{\mathcal{L}^\p}\|J\tilde{h}\|_{\mathcal{L}^\pprime}\] and using Corollary \ref{cor:J_op} in combination to \ref{thm:extrap} we conclude that \[\|J\tilde{h}\|_{\mathcal{L}^\pprime}\lesssim\|\tilde{h}\|_{\mathcal{L}^\pprime}.\]

Finally, again by H\"older's inequality we have that $\|\Lambda\|\lesssim \|h\|_{\mathcal{F}^\p}$ which implies that \[\|\Lambda\|\sim \|h\|_{\mathcal{F}^\p}.\]
\end{proof}

We finish this article with one consequence of the previous duality. Recall that if  $p\in\mathscr{P}(\CC)$, then $\mathcal{F}^{p^-}\subset\mathcal{F}^\p\subset \mathcal{F}^{p^+}$ and as an immediate corollary, we get that $\{K_z:z\in\CC\}\subset\mathcal{F}^\p$. Consequently,  every $f\in\mathcal{F}^\p$ satisfies the representation \eqref{eq:represent}.  Moreover,  denote as $V$ the closed vector subspace of $\mathcal{F}^\p$ generated by the set $\{K_z:z\in\CC\}$ and suppose that $f\in \mathcal{F}^\p\setminus V$, then there exists $h\in \mathcal{F}^{p'(\cdot)}$, $h\neq 0$ such that $\langle K_z, h\rangle=0$ for every $z\in \CC$, but then $h\equiv 0$, a contradiction. We record this in the following corollary.

\begin{cor}
Suppose $p\in\mathscr{P}(\CC)$ then the set of all linear combinations of reproducing kernels is dense in $\mathcal{F}^\p$.
\end{cor}
As a direct consequence, we have the density of the set of polynomials in $\mathcal{F}^\p$.
\begin{cor}
Suppose $p\in\mathscr{P}(\CC)$. Then the set of polynomials is dense in $\mathcal{F}^\p$.
\end{cor}
%
%\begin{proof}
%First, it is known that polynomials are dense in the constant exponent case (see for example \cite{FS}). In particular, they are dense in $\mathcal{F}^{p^-}\subset \mathcal{F}^\p$.
%Now suppose $f\in \mathcal{F}^\p$. By the previous corollary, given $\varepsilon>0$, it is possible to find a linear combination $g$  of reproducing kernels such that $\|f-g\|_{\mathcal{F}^\p}<\varepsilon$. Now, using the density of polynomials in $\mathcal{F}^{p^-}$, we choose a polynomial $h$ such that $\|g-h\|_{\mathcal{F}^{p^-}}<\varepsilon$. Finally, using Theorem \ref{thm:subset} we get that \[\|f-h\|_{\mathcal{F}^\p}\leq \|f-g\|_{\mathcal{F}^\p}+\|g-h\|_{\mathcal{F}^\p}\lesssim \|f-g\|_{\mathcal{F}^\p}+\|g-h\|_{\mathcal{F}^{p^-}}<2\varepsilon\] and the result follows. 
%\end{proof}

\section*{Acknowledgment}
The authors would like to thank David Cruz--Uribe for pointing out some references and properties of the Gaussian weight.

%
%%
%
%
%
%
%
%
%
%
%%% The Appendices part is started with the command \appendix;
%%% appendix sections are then done as normal sections
%%% \appendix
%
%%% \section{}
%%% \label{}
%
%%% References
%%%
%%% Following citation commands can be used in the body text:
%%% Usage of \cite is as follows:
%%%   \cite{key}          ==>>  [#]
%%%   \cite[chap. 2]{key} ==>>  [#, chap. 2]
%%%   \citet{key}         ==>>  Author [#]
%
%%% References with bibTeX database:
%
%\bibliographystyle{model1a-num-names}
%%% Authors are advised to submit their bibtex database files. They are
%%% requested to list a bibtex style file in the manuscript if they do
%%% not want to use model1a-num-names.bst.
%
%%% References without bibTeX database:
%
%
%

%\section*{References}
 \bibliographystyle{plain} 
   \bibliography{biblio2}

\end{document}